\documentclass{amsart}
\usepackage{amssymb}
\usepackage{amsmath}
\usepackage{verbatim}
\usepackage{hyperref}
\usepackage{amsrefs}

\usepackage{mathrsfs}

\def\Xint#1{\mathchoice
{\XXint\displaystyle\textstyle{#1}}%
{\XXint\textstyle\scriptstyle{#1}}%
{\XXint\scriptstyle\scriptscriptstyle{#1}}%
{\XXint\scriptscriptstyle\scriptscriptstyle{#1}}%
\!\int}
\def\XXint#1#2#3{{\setbox0=\hbox{$#1{#2#3}{\int}$}
\vcenter{\hbox{$#2#3$}}\kern-.5\wd0}}

\def\dashint{\Xint-}

\newcommand{\norm}[1]{\ensuremath{\left\|#1\right\|}}
\newcommand{\innp}[1]{\ensuremath{\left\langle#1 \right\rangle}}
\newcommand{\abs}[1]{\ensuremath{\left\vert#1\right\vert}}

\newcommand{\MC}[1]{\ensuremath{\mathcal{#1}}}

\newcommand{\D}{\mathscr D}
\newcommand{\pr}[1]{\ensuremath{{\left(#1\right)}}}

\newcommand{\R}{\mathbb{R}}

\newcommand{\F}{\ensuremath{\mathscr{F}}}
\renewcommand{\L}{\ensuremath{\mathscr{L}}}

\newcommand{\Cn}{\ensuremath{\mathbb{C}^n}}
\newcommand{\J}{\ensuremath{\mathscr{J}}}

\newcommand{\Rd}{\mathbb{R}^d}

\newcommand{\V}[1]{\ensuremath{\vec{#1}}}
\newcommand{\inrd}{\int_{\R^d}}

\newcommand{\Ap}{\ensuremath{\text{A}_p}}
\newcommand{\Apwk}{\ensuremath{\text{A}_p ^{\text{wk}}}}
\newcommand{\Apprimewk}{\ensuremath{\text{A}_{p'} ^{\text{wk}}}}

\newcommand{\br}[1]{\ensuremath{[#1]}}

\newcommand{\vertiii}[1]{{\left\vert\kern-0.25ex\left\vert\kern-0.25ex\left\vert #1
    \right\vert\kern-0.25ex\right\vert\kern-0.25ex\right\vert}}

\newcounter{vremennyj}

\numberwithin{equation}{section}

\newtheorem{thm}{Theorem}[section]

\title[Sharp matrix weighted estimates]{Sharp matrix weighted strong type inequalities for the dyadic square function}

\author{Joshua Isralowitz}
\begin{document}

\author{Joshua Isralowitz}
\address{Department of Mathematics, University at Albany, SUNY,
Albany, NY, 12159               }
\email{jisralowitz@albany.edu}

\maketitle

\begin{abstract}

In this paper we refine the recent sparse domination of the integrated $p = 2$ matrix weighted dyadic square function by T. Hytonen, S. Petermichl, and A. Volberg to prove a pointwise sparse domination of general matrix weighted dyadic square functions. We then use this to prove sharp two matrix weighted strong type inequalities for matrix weighted dyadic square functions when $1 < p \leq 2$. \end{abstract}

\section{Introduction} Let be $U$ an a.e. positive definite $n \times n$ matrix valued function on $\Rd$ (that is, a matrix weight), and for a measurable $\Cn$ valued function $\V{f}$ on $\Rd$ define \begin{equation*} \|\V{f}\|_{L^p(U)} := \left(\inrd \abs{U^\frac{1}{p} (x) \V{f}(x)}^p \, dx \right)^\frac{1}{p} \end{equation*} where $\abs{\V{e}}$ is the standard Euclidean norm on $\Cn$.  We will say that a pair of matrix weights $U, V$ is matrix A${}_p$  if \begin{equation*} \br{U,V}_{\Ap} :=  \sup_{\substack{I \subseteq \Rd \\ I \text{ is a cube }}} \dashint_I \left( \dashint_I \| V^{-\frac{1}{p}}  (y) U ^\frac{1}{p} (x)\| ^{p'} \, dy \right)^\frac{p}{p'} \, dx < \infty \end{equation*}  where $\dashint_I$  refers to the unweighted average and $\|A\|$ is the standard matrix norm of an $n \times n$ matrix $A$.  Clearly this is a condition that reduces to the classical Muckenhoupt two weight A${}_p$ condition in the scalar setting (when $n = 1$). If $U = V$ then we will say $U$ is a matrix A${}_p$ weight if $\br{U}_{\Ap} := \br{U,U}_{\Ap} < \infty$.   While it is known that most ``classical" operators from harmonic analysis (such as the maximal function, Calder\'{o}n-Zygmund operators,  paraproducts,  martingale transforms, square functions, etc.) are bounded on $L^p(U)$ for matrix A${}_p$ weights $U$, it is difficult to determine the sharp dependence of such operators on $\br{U}_{\Ap}$.

In fact, the only two such operators where sharp one weighted matrix weighted norm inequalities for $p = 2$ are known are for the dyadic square function, which was recently proved in \cite{HPV} and the maximal function, which was proved in \cite{IKP} by slightly modifying the ideas in \cite{CG}.   Furthermore, among these two operators, sharp one matrix weighted A${}_p$ bounds for $1 < p < \infty$ are only known for the maximal function, which were proved in \cite{IM} by slightly modifying the ideas in \cite{Go}.

The purpose of this paper is to prove sharp strong type matrix weighted norm inequalities for the dyadic square function in the range $1 < p \leq 2$, providing the first sharp $p \neq 2$ estimates for a singular operator in the matrix weighted setting.  Let $\D$ be a dyadic grid and let $\{h_J ^k\}$ for $k = 1, \ldots, 2^d - 1$ and $J \in \D$ be any Haar system on $\Rd$, meaning that $\{h_J ^k\}$ is an orthonormal system of $L^2(\Rd)$ with $h_J ^k$ supported on $J, \int_J h_J ^k(x) \, dx  = 0 $, and each $h_J ^k$ is constant on dyadic subcubes of $J$. Also, for a function $\V{f} : \Rd \rightarrow \Cn$ let \begin{equation*} \V{f}_J ^k := \int_J \V{f}(x) h_J ^k (x) \, dx \end{equation*}
and define the matrix weighted dyadic square function $S_U \V{f} = S_{U, p} \V{f}$ by \begin{equation} S_{U, p} \V{f}(x) := \left( \sum_{J, k} \frac{\abs{U^\frac{1}{p} (x)  \V{f}_J ^k}^2 1_J (x) }{|J|} \right)^\frac{1}{2}. \label{SquareFunction} \end{equation}   For notational ease we will omit the dependence of $h_J ^k$ on $k$ and presume all sums involving Haar functions are taken over $k = 1, \ldots, 2^d - 1$.  Note that this operator is a natural substitute for the dyadic square function in the sense that if $S_d$ is the ordinary dyadic square  function on scalar valued functions then \begin{equation*} \norm{S_d f}_{L^p(v) \rightarrow L^p(u)}  = \norm{S_{u, p} f}_{L^p (v) \rightarrow L^p}.\end{equation*}

  To state our main result we need the following definition.    We say that a matrix weight $U$ is matrix A${}_p ^\text{wk}$ if \begin{equation*} [U]_{\Apwk} := \sup_{\V{e} \in \Cn} \norm{\abs{U^\frac{1}{p} \V{e}}^p }_{\text{A}_\infty} < \infty. \end{equation*}  It is easy to show (see \cite{Go} for example) that a matrix A${}_p$ weight is also a matrix A${}_p ^\text{wk}$ weight with $[U]_{\Apwk} \leq  [U]_{\Ap}$ and clearly in the scalar setting we have $[u]_{\Apwk} = [u]_{\text{A}_\infty}$. Our first result is the following.

  \begin{thm} \label{MainThmNormSquareFunction} If $U, V$ is a pair of matrix A${}_p$ weights, $V^{-\frac{p'}{p}}$ is a matrix A${}_{p'} ^ \text{wk}$ weight, and $1 < p \leq 2$  then the sharp estimate\begin{equation*} \|S_{U, p} \|_{L^p (V) \rightarrow L^p} \lesssim \br{U,V}_{\Ap}^\frac{1}{p}  [V^{-\frac{p'}{p}}]_{\Apprimewk} ^\frac{1}{p} \end{equation*} holds.   \\

  Furthermore, if $2 < p < \infty$ we have the following (most likely not sharp) estimate \begin{equation*} \|S_{U, p} \|_{L^p(V) \rightarrow L^p} \lesssim \br{U,V}_{\Ap}^\frac{1}{p}  \br{V^{-\frac{p'}{p}}}_{\Apprimewk} ^\frac{1}{p} \br{U}_{\Apwk} ^{\frac{1}{2} - \frac{1}{p}} \end{equation*} \end{thm}

  Note that this was proved when $p = 2$ in \cite{HPV} in the one weighted case and that sharpness when $1 < p \leq 2$ follows from the well known sharpness in the scalar setting (see \cite{LL,HL}). Also, note that while it is unlikely that Theorem \ref{MainThmNormSquareFunction} is sharp when $p > 2$, it is a natural bound and in fact we will recover from the proof of Theorem \ref{MainThmNormSquareFunction} the current best mixed matrix weighted A${}_p$ - A${}_\infty$ bound for a positive sparse operator $\tilde{\MC{S}}_{U}$ from \cite{CIM} (and thus the current best bound for CZOs via the sparse convex body domination theorem from \cite{NPTV}), namely \begin{equation*} \norm{\tilde{\MC{S}_{U}}}_{L^p \rightarrow L^p} \lesssim \br{U,V}_{\Ap} ^\frac{1}{p} \br{V^{-\frac{p'}{p}}} _{\Apprimewk} ^\frac{1}{p} \br{U}_{\Apwk} ^\frac{1}{p'}   \end{equation*} (see p. 9 for the definition of a positive Sparse operator $\tilde{\MC{S}}_{U}$).

We will now outline the arguments used to prove our main result.  In the next section, we will modify the stopping time ideas from \cite{HPV} to prove a sparse domination of $S_{U, p}\V{f}(x)$ for all matrix weights $U,  \ 1 < p < \infty,$ and $\V{f}$ measurable.  Note that unlike in \cite{HPV} which proved the sparse domination of an integrated version of $S_{U,2} \V{f}$, we will actually prove a sparse \textit{pointwise} domination of $S_{U, p}\V{f}(x)$.  In the third section we will prove Theorem \ref{MainThmNormSquareFunction} by ``matrixizing" some of the ideas in \cite{C} to prove a matrix weighted Carleson embedding type theorem.  Of particular novelty here is that we will use a matrix weighted ``stopping moment"  decomposition, which to the author's knowledge is the first time such an argument in the matrix weighted setting has appeared. Note that a similar matrix weighted parallel corona decomposition argument should be possible (which in fact was used to prove a sharp version of Theorem \ref{MainThmNormSquareFunction} in the scalar $p  > 2$ setting in \cite{LL}).

We will end this paper with an important point.  First, as of the date of writing this paper, it is unknown whether the Rubio de Francia extrapolation theorem holds.  Namely, it is not known whether the boundedness of an operator $T$ on $L^2(U)$ for all matrix A${}_2$ weights $U$ implies the boundedness of $T$ on $L^p(U)$ for all matrix A${}_p$ weights $U$ and all $1 < p < \infty.$  Thus, unlike in the scalar setting, sharp estimates (or even just boundedness) of operators for $p \neq 2$ do not at this moment follow from sharp estimates of operators for $p  = 2$.

\section{Sparse domination of square functions}
Before we state the main result of this section we will need to introduce some definitions and notation.  First we will introduce the concept of a reducing matrix, whose importance was emphasized in \cite{Go} and which has since shown to be vital in the theory of matrix weighted norm inequalities.  Namely, for a matrix weight $U$, a cube $I$, and   $\V{e} \in \Cn$ there exists positive definite matricies $\MC{U}_I, \MC{U}_I '$ where \begin{equation*} \abs{\MC{U}_I \V{e}} \approx \pr{\dashint_I \abs{U^\frac{1}{p} (x) \V{e}}^p \, dx }^\frac{1}{p}, \qquad \abs{\MC{U}_I ' \V{e}} \approx \pr{\dashint_I \abs{U^{-\frac{1}{p}} (x) \V{e}}^{p'} \, dx }^\frac{1}{p'} \end{equation*} where the implicit constant depends only on $n$.  In particular, it is easy to see that \begin{equation*} \br{U,V}_{\Ap} \approx \sup_{\substack{I \subseteq \Rd \\ I \text{ is a cube }}} \norm{\MC{U}_I \MC{V}_I '}^p. \end{equation*}  \noindent Now let $\{\V{e}_j\}_{j = 1} ^n$ be any orthonormal basis of $\Cn$.  We will then use the following simple estimate without further mention throughout the rest of the paper:  If $A$ is any $n \times n$ matrix then \begin{equation*} \|\MC{U} _I A\| ^p \approx \sum_{j = 1}^n \abs{\MC{U} _I A \V{e}_j} ^p \approx \sum_{j = 1}^n \dashint_I \abs{U^\frac{1}{p} (x) A \V{e}_j} ^p \, dx \approx  \dashint_I \norm{U^\frac{1}{p} (x) A} ^p \, dx. \end{equation*}

 Let $\D$ be a dyadic grid.  A collection $\L$ of dyadic cubes in $\D$ is sparse if

\begin{equation*} \bigcup_{\substack{L  \varsubsetneq J \\ L, J \in \L}} |L| \leq \frac12 |J|. \end{equation*} See \cite{LN} or \cite{NPTV} for more properties of sparse collections. \\

Given a sparse collection $\L$, define the ``sparse positive operator" $\tilde{S}_{U, \L} = \tilde{S}_{U,  p, \L}$ by \begin{equation} \tilde{S}_{U, \L} \V{f} (x) := \left( \sum_{L \in \L}  \innp{|\MC{U}_L  \V{f}|} _L ^2 \norm {U^\frac{1}{p} (x) \MC{U}_L^{-1} } ^2 1_L (x) \right)^\frac{1}{2} \label{PositiveSparse}\end{equation} where $\innp{ }_L$ denotes the unweighted average over $L$.   Furthermore, for any $J \in \L$ define the localized sparse positive operator $ \tilde{S}_{U, J, \L}$ by \begin{equation*} \tilde{S}_{U, J, \L} \V{f} (x) := \left( \sum_{\substack{L \in \L \\ L \subseteq J}}  \innp{|\MC{U}_L  \V{f}|} _L ^2 \norm{U^\frac{1}{p} (x) \MC{U}_L^{-1} } ^2 1_L (x) \right)^\frac{1}{2}\end{equation*} and similarly define $S_{U, J} \V{f}$ by \begin{equation*} S_{U, J}  \V{f}(x) = \left( \sum_{\substack{L \in \D  \\ L \subseteq J }} \frac{\abs{U^\frac{1}{p} (x)  \V{f}_L}^2 1_L (x) }{|L|} \right)^\frac{1}{2}\end{equation*} Finally, for $N \in \mathbb{N}, $ define

\begin{equation} \label{SquareFunctionN} S_{U, N}  \V{f}(x) = \left( \sum_{\substack{L \in \D  \\ 2^{-N} \leq \ell(L) \leq 2^N }} \frac{\abs{U^\frac{1}{p} (x)  \V{f}_L}^2 1_L (x) }{|L|} \right)^\frac{1}{2}\end{equation} where $\ell(L)$ is the sidelength of $L$, and define $S_{U, J, N}$ in an analogous way.

The main result of this section is the following.
 \begin{thm} \label{SparseDomination} Let $S_{U, N} \V{f}$ denote the matrix weighted square function defined by \eqref{SquareFunctionN}.  Then there exists a sparse collection $\L$ of dyadic cubes where for a.e. $x \in \Rd$ we have $S_{U, N} \V{f}(x) \lesssim  \tilde{S}_{U, \L} \V{f}(x)$ with implicit constant independent of $\V{f}, x, N, $ and $ U$.   \end{thm}

 \begin{proof} Let $S_d$ denote the unweighted dyadic square function with respect to $\D$. It is then enough to show that for each $J \in \D$ we can find a sparse collection $\L$ of dyadic subcubes of $J$ where $S_{U,  J, N}  \V{f}(x) \lesssim  \tilde{S}_{U, J, \L} \V{f}(x)$, since then we can apply this to each $J \in \D$ with $\ell(J) = 2^N$.   Let $\J(J)$ be the maximal dyadic subcubes $L$ of $J$ where  \begin{equation*} \sum_{ J \supseteq I \supseteq L} \frac{|\MC{U}_J \V{f}_I|^2}{|I|} > \lambda \innp{|\MC{U}_J\V{f}|}_J ^2 \end{equation*}
 We claim that $\sum_{L \in \J(J)} |L| \leq \frac14 |J|$ for large enough $\lambda$.  For that matter,  if $x \in L \in \J(J)$ then \begin{align*} S_d (\MC{U}_J 1_J \V{f}) (x) ^ 2  & = \sum_{I \in \D} \frac{|\MC{U}_J (1_J\V{f})_I|^2}{|I|} 1_I (x)
 \\ & \geq \sum_{J \supseteq I \supseteq L} \frac{|(\MC{U}_J \V{f} 1_J)_I|^2}{|I|}
 \\ & =  \sum_{ J \supseteq I \supseteq L} \frac{|\MC{U}_J \V{f}_I|^2}{|I|}
 \\ & > \lambda \innp{|\MC{U}_J \V{f}|}_J ^2 \end{align*} so that using the fact that $\|S_d\|_{L^{1} \rightarrow L^{1, \infty}} \leq C$ we get \begin{equation*} \sum_{L \in \J_1 (J)} |L| = \abs{\bigcup_{L \in \J_1(J)} L }  \leq |\{x : S_d(\MC{U}_J 1_J \V{f} ) (x) \geq \lambda ^\frac12 \innp{|\MC{U}_J \V{f}|}_J \}| \leq \frac{C}{\lambda^\frac12} |J| \end{equation*} which clearly proves the claim.

 Let $\F(J)$ denote the collection of all $L \in \D(J)$ such that $L \not \subseteq Q$ for any $Q \in \J(J)$. Furthermore, abusing notation slightly, we will denote $\cup_{Q \in \J(J)} Q$ by $\cup\J(J)$. Fix $x \in \J$ such that $U(x)$ is defined. Then \begin{align*} \sum_{L \in \D(J)}  \frac{\abs{U^\frac{1}{p} (x) \V{f}_L }^2 1_L (x) }{|L|}
  & = \sum_{L \in \F(J)} \frac{\abs{U^\frac{1}{p} (x) \V{f}_{L} } ^2 1_{L  } (x)}{|L|}
+ \sum_{Q \in \J(J)} \sum_{L \in \D(Q)}  \frac{\abs{U^\frac{1}{p} (x) \V{f}_L }^2 1_L (x) }{|L|}
\\ & = A(x) + \sum_{Q \in \J(J)} \sum_{L \in \D(Q)}  \frac{\abs{U^\frac{1}{p} (x) \V{f}_L }^2 1_L (x) }{|L|}\end{align*}

  We estimate $A(x)$ by considering two cases.  First assume $x \in \cup\J(J).$ Thus, if $x \in I$ for some $I \in \J(J)$ and $x \in L \in \F(J)$ then again by definition of $\F(J)$ we have $J \supseteq L \varsupsetneq I$  so that

  \begin{align*} A(x) &  \leq \norm{ U^\frac{1}{p} (x)\MC{U}_J ^{-1} }^2 1_J (x)  \sum_{J \supseteq L \varsupsetneq I} \frac{\abs{ \MC{U}_J \V{f}_{L} } ^2 }{|L|}
  \\ & \leq  \lambda \norm{ U^\frac{1}{p} (x) \MC{U}_J ^{-1}}^2 1_J (x) \innp{ | \MC{U}_J \V{f} |}_J ^2. \end{align*}

  On the other hand, if $x \not \in  \cup \J(J)$  then we can pick a sequence of nested dyadic cubes $\{L_k ^x\} = \{L \in \F(J) : x \in L\} = \{L \in \D(J) : x \in L\}$.    However, if \begin{equation*} \sum_k  \frac{|\MC{U}_J \V{f}_{L_k ^x} |^2}{|L_k^x|} > \lambda \innp{|\MC{U}_J\V{f}|}_J ^2 \end{equation*} then obviously for some $k'$ we must have

  \begin{equation*} \sum_{ J \supseteq L \supseteq L_{k'} ^ x } \frac{|\MC{U}_J \V{f}_L|^2}{|L|} > \lambda \innp{|\MC{U}_J\V{f}|}_J ^2 \end{equation*} which means $x \in L_{k'} ^x \subseteq I$ for some $I \in \J(J)$.  Thus,

  \begin{align*}  A(x) & \leq \norm{U^\frac{1}{p} (x) \MC{U}_J ^{-1}}^2  \sum_{L \in \F(J)} \frac{\abs{\MC{U}_J  \V{f}_L   }^2 1_{L}  (x)}{|L|}
  \\ & \leq  \norm{U^\frac{1}{p} (x) \MC{U}_J ^{-1}} ^2 \sum_{k}   \frac{\abs{\MC{U}_J  \V{f}_{L_k ^x}    }^2 }{|L_k ^x|}
  \\ & \leq \lambda  \norm{U^\frac{1}{p} (x) \MC{U}_J ^{-1}}^2 \innp{|\MC{U}_J\V{f}|}_J ^2 1_J (x)
  \end{align*}

Putting this together, we get \begin{align*} \sum_{L \in \D(J)}  \frac{\abs{U^\frac{1}{p} (x) \V{f}_L }^2 1_L (x) }{|L|}  \leq \lambda    \norm{U^\frac{1}{p} (x) \MC{U}_J ^{-1}}^2 \innp{|\MC{U}_J \V{f}|}_J ^2 1_J (x) + \sum_{Q \in \J(J)} \sum_{L \in \D(Q)}  \frac{\abs{U^\frac{1}{p} (x) \V{f}_L }^2 1_L (x) }{|I|}. \end{align*} Finally set $\J_0(J) = \{J\}$ and for $k \in \mathbb{N}$ set $\J_k(J) = \{L \in \J(Q) : Q \in \J_{k-1}(J)\}$.  If $\L  = \cup_k \J_k(J)$ then $\L$ is sparse, which by iteration completes the proof of Theorem \ref{SparseDomination}
 \end{proof}

\section{Proof of Theorem \ref{MainThmNormSquareFunction}} \label{NormSquareSection}

In this section we will prove Theorem \ref{MainThmNormSquareFunction} by utilizing Theorem \ref{SparseDomination}.   Again fix $\lambda > 1$ to be determined momentarily.  Given $Q \in \D$ let $\MC{G}(Q)$ denote the set of maximal $L \in \D(J)$ such that either \begin{equation}  \dashint_L |\MC{U}_J  \vec{f}|> \lambda \dashint_J|\MC{U}_J \vec{f}| \label{StopOne} \end{equation} or \begin{equation} \| \MC{U}_L \MC{U}_J ^{-1} \| > \lambda \label{StopTwo}\end{equation}   We now prove that for $\lambda > 0$ large enough we have that \begin{equation} \label{sparseineq} \sum_{L \in \MC{G}(J)} |L| \leq \frac14 |J| \end{equation}

Let $\J_1(J)$ and $\J_2(J)$ denote those maximal cubes in $J$ satisfying \eqref{StopOne} and \eqref{StopTwo}, respectively.  Then, as usual,  \begin{equation*} \sum_{L \in \J_1(J)} |L| \lesssim \frac{1}{\lambda^p } \sum_{L \in \J_1(J)} \int_L \norm{U^\frac{1}{p} (x) \MC{U}_J ^{-1}} ^p \, dx \lesssim \frac{1}{\lambda^p } \int_J \norm{U^\frac{1}{p} (x) \MC{U}_J ^{-1}} ^p \, dx \lesssim \frac{1}{\lambda^p }  |J| \end{equation*}  and furthermore \begin{align*} \sum_{L \in \J_2(J)} |L| & \leq \frac{1}{\lambda} \sum_{L \in \J_2(J)}  |L| \frac{ \dashint_L | \MC{U}_J  \vec{f}| }{\dashint_J | \MC{U}_J \vec{f}|}
\\ & \leq \frac{1 }{\lambda} |J|. \end{align*} This completes the proof for $\lambda$ large enough, since $\J(J) \subseteq \J_1(J) \cup \J_2(J)$.

Now for fixed $N \in \mathbb{N}$ let \begin{equation*} \MC{G}_0 = \{J \in \D : |J| = 2^{N}\}\end{equation*} and inductively define \begin{equation*} \MC{G}_{k + 1} = \{L \in \D : L \in \MC{G}(J) \text{ for some } J \in \MC{G}_k\}. \end{equation*} If $\MC{E}(J)$ denotes the collection of all $L' \in \D(J)$ that are not contained in any $L \in \MC{G}(J)$ and $\MC{G}$ is the union \begin{equation*} \MC{G} = \bigcup_{k = 0}^\infty \MC{G}_k \end{equation*} then we clearly have \begin{equation} \bigcup_{J \in \MC{G}} \MC{E}(J)  = \{J \in \D : |J| \leq  2^{N}\} \label{DPart}\end{equation} for $\lambda > 0$ large enough since $J \in \MC{E}(J)$ for any $J \in \D$.  Also clearly an iteration of \eqref{sparseineq} gives us that for any $Q \in \MC{G}$ \begin{equation} \sum_{L \in \MC{G} , L \subseteq Q} |Q| \leq \frac12 |L| \label{SparseProp}. \end{equation}

\noindent Furthermore, it is important to note that if $L \in \MC{E}(J)$ then both \eqref{StopOne} and \eqref{StopTwo} are false, so that \begin{equation} \dashint_L |\MC{U}_L  \vec{f}| \leq \|\MC{U}_L \MC{U}_J^{-1} \| \dashint_L |\MC{U}_J  \vec{f}| \leq \lambda ^2  \dashint_J |\MC{U}_J  \vec{f}| \label{CoronaEq} \end{equation}

We now state and prove a Carleson embedding type theorem for the type of operator used in the previous section, which will easily show Theorem \ref{MainThmNormSquareFunction}.  Given nonnegative measurable  functions $\{a_L (x) \}_{L \in \D}$  and $r > 0$, define $\tilde{S}_{U, a} \V{f} = \tilde{S}_{U, a,  r, p} \vec{f}$ by \begin{equation*} \tilde{S}_{U, a} \V{f} (x) = \left(\sum_{L \in \D} a_L (x) { \innp{\abs{   \MC{U}_L  \vec{f}}}_L ^r 1_L(x)} \right)^\frac{1}{r}. \end{equation*}

\begin{thm} \label{MaxEmbThm} Let $1 < p \leq r$, let $V^{-\frac{p'}{p}}$ be a matrix A${}_{p'} ^ \text{wk}$ weight and let
\begin{align*}\|A\|_* =  \sup_{J \in \D} \dashint_J  \left(\sum_{L \in \MC{D}(J)}  a_L (x) { 1_L(x)} \right)^\frac{p}{r}\, dx. \end{align*}

then \begin{equation*} \|\tilde{S}_{U, a} \|_{L^p(V) \rightarrow L^p} \lesssim [U, V]_{\text{A}_p} ^\frac{1}{p } \br{V^{-\frac{p'}{p}}}_{\Apprimewk} ^\frac{1}{p}   \|A\|_* ^\frac{1}{p}. \end{equation*} \label{MaxEmbThm}

\end{thm}

\begin{proof}  Let \begin{equation*} F_J (x) = \left(\sum_{L \in \MC{E}(J)} a_L(x) {  \innp{\abs{\MC{U}_L   \vec{f}}_L }^r}  1_L(x) \right)^\frac{1}{r}. \end{equation*}  We first get a bound for $\|F_J\|_{L^p}$.  Note that \begin{align*}
F_J (x) &= \left(\sum_{L \in \MC{E}(J)}  a_L(x) { \innp{ \abs{  \MC{U}_L  \vec{f}}_L }^r 1_L(x)} \right)^\frac{1}{r}
\\ & \lesssim  \innp{ \abs{ \MC{U}_J  \vec{f}}}_J \left(\sum_{L \in \MC{E}(J)}  a_L(x)  1_L(x)  \right)^\frac{1}{r} \end{align*} by \eqref{CoronaEq} since $L \in \MC{E}(J)$.  Thus,

\begin{align*} \|F_J\|_{L^p} ^p & \lesssim \innp{ \abs{ \MC{U}_J  \vec{f}}}_J  ^p  \int_J   \left(\sum_{L \in \MC{E}(J)}  a_L(x){1_L(x)}  \right)^\frac{p}{r} \\ & \leq \|A\|_* |J|  \innp{\abs{\MC{U}_J  \vec{f}}}_J ^p  . \end{align*}


However, if \begin{equation*} \tilde{S}_{U, a, N} \vec{f}(x) = \left(\sum_{|L| \leq 2^{N}} a_L(x) { \innp{\abs{ \MC{U}_L  \vec{f}} }_L ^r 1_L (x)} \right)^\frac{1}{r} \end{equation*} then \eqref{DPart} gives us that \begin{equation*}\tilde{S}_{U, a, N} \V{f}(x) = \left(\sum_{J \in \MC{G}} F_J ^r (x) \right)^\frac{1}{r}. \end{equation*}

Then using the fact that $p \leq r$,  \begin{align} \|\tilde{S}_{U, a, N} \V{f} \|_{L^p} & = \left\|\left(\sum_{J \in \MC{G}} F_J  ^r \right)^\frac{1}{r} \right\|_{L^p} \nonumber  \\ & \leq \left\|\left(\sum_{Q \in \MC{G}} F_Q ^p \right)^\frac{1}{p} \right\|_{L^p} \nonumber  \\ & = \left( \sum_{Q \in \MC{G}} \|F_Q \|_{L^p} ^p \right)^\frac{1}{p} \nonumber \\ & \lesssim \|A\|_*^\frac{1}{p} \left( \sum_{J \in \MC{G}} |J|  \innp{\abs{\MC{U}_J  \vec{f}}}_J ^p\right)^\frac{1}{p}. \label{MaxEmbThmEst} \end{align}

By the sharp reverse H\"{o}lder inequality for A${}_\infty$ weights, we can pick $\epsilon \approx [V^{-\frac{p'}{p}}]_{\text{A}_{p'
 }^\text{wk}}^{-1}$ small enough where

\begin{align*} \dashint_J  |\MC{U}_J  \vec{f}|  & \leq \left(\dashint_J  \| \MC{U}_J V^{-\frac{1}{p}}\|^\frac{p - \epsilon }{p - \epsilon - 1} \right)^\frac{p - \epsilon - 1 }{p - \epsilon} \left(\dashint_J |V^\frac{1}{p} \vec{f}|^{p - \epsilon}\right)^\frac{1}{p - \epsilon}
\\ & \lesssim \left(\dashint_J  \| \MC{U}_J V^{-\frac{1}{p}}  \|^{p'} \right)^\frac{1 }{p'}  \left(\dashint_J |V^\frac{1}{p} \vec{f}|^{p - \epsilon}\right)^\frac{1}{p - \epsilon}
\\ & \lesssim [U, V]_{\text{A}_p} ^\frac{1}{p} \left(\dashint_J |V^\frac{1}{p} \vec{f}|^{p - \epsilon}\right)^\frac{1}{p - \epsilon}. \end{align*}

Now, for any nonnegative scalar Carleson sequence $(\tau_Q)$, if $$\|\tau_Q\|_* = \sup_{J \in \D} \frac{1}{|J|} \sum_{L \in \D(J)} \tau_L$$ the standard proof of the (unweighted) dyadic Carleson embedding theorem and the well known ``$L^{1+\delta}$" Maximal function bound for $\delta>0$ small tells us that for $q = 1+\delta$ \begin{align*} \sum_{Q \in \D} \tau_Q \innp{|f|}_Q ^q \leq \|\tau_Q\|_* \|M_d f\|_{L^q} ^q \lesssim \delta^{-1} \|\tau_Q\|_* \|f\|_{L^q} ^q  \end{align*}  Applying this to the exponent $\frac{p}{p - \epsilon} > 1$ ,  \eqref{MaxEmbThmEst} gives us \begin{align*} \norm{ \tilde{S}_{U, a, N} \V{f}}_{L^p} & \lesssim  \|A\|_* ^\frac{1}{p} [U, V]_{\text{A}_p} ^\frac{1}{p}    \left(\sum_{J \in \MC{G}}  |J|        \innp{ |V^\frac{1}{p} \vec{f}|^{p - \epsilon}}_J ^\frac{p}{p  - \epsilon}  \right)^\frac{1}{p}
\\ & \lesssim \|A\|_*^\frac{1}{p} [U, V]_{\text{A}_p} ^\frac{1}{p}   \epsilon ^{-\frac{1}{p}}  \left(\inrd (|V ^\frac{1}{p} \vec{f}(x)|^{p - \epsilon}) ^\frac{p }{p - \epsilon} \, dx \right)^\frac{1}{p} \\ & = \|A\|_* ^\frac{1}{p} [U, V]_{\text{A}_p} ^\frac{1}{p}    [V^{-\frac{p'}{p}}]_{\text{A}_{p'} ^\text{wk}} ^{\frac{1}{p}}  \|\vec{f}\|_{L^p(V)}. \end{align*} Letting $N \rightarrow \infty$ in conjunction with the monotone convergence theorem completes the proof. \end{proof}

Finally, to see how this proves Theorem \ref{MainThmNormSquareFunction} when $1 < p \leq 2$, set $r = 2$ and let $\L $ be a sparse collection.  Set \[a_L(x) =  \begin{cases}
      \norm{U^\frac{1}{p} (x) \MC{U}_L^{-1} }^2 & L \in \L \\
      0 & L \not \in \L
   \end{cases}
\]

  Then since $\frac{p}{2} \leq 1$ \begin{align*} \sup_{J \in \MC{D}} \dashint_J  \left(\sum_{\substack {L \subseteq J  \\ L \in \L}}   { \norm{U^\frac{1}{p} (x) \MC{U}_L^{-1} }^2 1_L(x)} \right)^\frac{p}{2}\, dx  &
  \leq \sup_{J \in \MC{D}} \frac{1}{|J|}   \sum_{\substack {L \subseteq J  \\ L \in \L}}   \int_L  \norm{U^\frac{1}{p} (x) \MC{U}_L^{-1} }^p \, dx
  \\ & \lesssim \sup_{J \in \MC{D}} \frac{1}{|J|}   \sum_{\substack {L \subseteq J \\ L \in \L}}  |L|
  \\ & = \sup_{J \in \MC{D}} \frac{1}{|J|}   \sum_{L^* \subseteq J} \sum_{\substack {L \subseteq L^* \\ L \in \L}}  |L|
  \\ & \leq \frac32 \sup_{J \in \MC{D}} \frac{1}{|J|} \sum_{L^* \subseteq J} |L^*| \leq \frac32
  \end{align*}

  \noindent where here $\{L^*\}$ is the collection of maximal $L^* \in \L$ with $L^* \subseteq J$.

Thus, if $\tilde{S}_{U, \L}$ is defined as in \eqref{PositiveSparse}, then Theorem \ref{MaxEmbThm} gives us that for $r = 2$\begin{equation*} \|\tilde{S}_{U, \L} \|_{L^p(V) \rightarrow L^p} = \|\tilde{S}_{U, a}\|_{L^p(V) \rightarrow L^p} \lesssim [U, V]_{\text{A}_p} ^\frac{1}{p}    [V^{-\frac{p'}{p}}]_{\text{A}_{p'} ^\text{wk}} ^{\frac{1}{p}}. \end{equation*}  But Theorem \ref{SparseDomination} and the monotone convergence theorem then says for any $\V{f} \in L^p(V)$, we have that \begin{align*} \| S_U \V{f} \|_{L^p} & = \lim_{N \rightarrow \infty} \| S_{U, N} \V{f} \|_{L^p}
\\ & \lesssim  \|\tilde{S}_{U, \L} \V{f}\|_{ L^p}
\\ & \lesssim [U, V]_{\text{A}_p} ^\frac{1}{p}    [V^{-\frac{p'}{p}}]_{\text{A}_{p'} ^\text{wk}} ^{\frac{1}{p}}  \| \vec{f}\|_{L^p(V)}.  \end{align*}

 To prove Theorem \ref{MainThmNormSquareFunction} when $p > 2$, we argue as in \cite{CIM} and use a routine duality argument.  In fact, we will prove a slightly stronger result.  Given a sparse collection $\L$ and $r > 0$, define the  $\tilde{S}_U = \tilde{S}_{U, \L, r, p} $ by \begin{equation*} \tilde{S}_U \V{f} (x) := \left( \sum_{L \in \L}  \innp{|\MC{U}_L  \V{f}|} _L ^r \norm {U^\frac{1}{p} (x) \MC{U}_L^{-1} } ^r 1_L (x) \right)^\frac{1}{r}. \end{equation*}  Assume $p > r$ so that $\frac{p}{r} > 1$.  Then \begin{align*} \norm{\tilde{S}_U \V{f} } _{L^p} & = \norm{\sum_{L \in \L}  \innp{|\MC{U}_L  \V{f}|} _L ^r \norm {U^\frac{1}{p}  \MC{U}_L^{-1} } ^r 1_L }_{L^{\frac{p}{r}}} ^\frac{1}{r}
 \\ & \lesssim \sup_{\|g\|_{L^{\frac{p}{p-r}} \leq 1}} \pr{ \sum_{L \in \L}  |L| \innp{|\MC{U}_L  \V{f}|} _L ^r  \innp{\norm {U^\frac{1}{p}  \MC{U}_L^{-1} } ^r g }_L   }^\frac{1}{r}.
 \end{align*}

However, as in the $1 < p \leq 2$ case, by the sharp reverse H\"{o}lder inequality for A${}_\infty$ weights we can pick $\epsilon_1 \approx \br{V^{-\frac{p'}{p}}}_{\Apprimewk}^{-1}$ and $\epsilon_2 \approx \br{U}_{\Apwk} ^{-1}$ where

\begin{equation*} \innp{|\MC{U}_L  \V{f}|} _L ^r      \leq \innp{\|\MC{U}_L V^{-\frac{1}{p}}\|^{\frac{p-\epsilon_1}{p-\epsilon_1 - 1}}}_L ^{r\pr{\frac{p-\epsilon_1 - 1}{p-\epsilon_1}}}\innp{|V^{\frac{1}{p}}\V{f}|^{p-\epsilon_1}}_L ^{\frac{r}{p-\epsilon_1}}      \lesssim \br{U,V}_{\Ap} ^\frac{r}{p}  \innp{|V^{\frac{1}{p}}\V{f}|^{p-\epsilon_1}}_L ^\frac{r}{p-\epsilon_1} \end{equation*}

and

\begin{equation*}  \innp{\norm {U^\frac{1}{p}  \MC{U}_L^{-1} } ^r g }_L \leq \innp{\norm{U^\frac{1}{p}  \MC{U}_L^{-1} } ^{r\pr{\frac{p-\epsilon_2}{r-\epsilon_2}}}}_L ^\frac{r-\epsilon_2}{p-\epsilon_2}  \innp{|g|^\frac{p-\epsilon_2}{p - r} }_L ^ \frac{p - r}{p-\epsilon_2}
\lesssim \innp{|g|^\frac{p-\epsilon_2}{p - r} }_L ^ \frac{p - r}{p-\epsilon_2} \end{equation*}

If as usual \begin{equation*} E_L = L \backslash \bigcup_{\substack{L' \subsetneq  L \\ L' \in \L}} L' \end{equation*} then the sets $\{E_L : L \in \L\}$ are disjoint and $|L| \leq 2 |E_L|$.  We then have
\begin{align*} \br{U,V}_{\Ap} ^\frac{1}{p} &  \pr{\sum_{L \in \L} |L|  \innp{|V^\frac{1}{p} \V{f}|^{p-\epsilon_1}}_L ^\frac{r}{p-\epsilon_1} \innp{|g|^{\frac{p-\epsilon_2}{p - r}} }_L ^ {\frac{p - r}{p-\epsilon_2}}} ^\frac{1}{r}
\\ & \lesssim \br{U,V}_{\Ap} ^\frac{1}{p}  \pr{\sum_{L \in \L} \int_{E_L}  \pr{M_d |V^\frac{1}{p}\V{f}|^{p-\epsilon_1} (x) } ^{\frac{r}{p-\epsilon_1}} \pr{ M_d    |g|^{\frac{p-\epsilon_2}{p - r}}  (x)}^ {\frac{p - r}{p-\epsilon_2}} \, dx } ^\frac{1}{r}
\\ & \leq \br{U,V}_{\Ap} ^\frac{1}{p} \pr{ \inrd (M_d |V^\frac{1}{p}\V{f}|^{p-\epsilon_1})^\frac{p}{p-\epsilon_1}}^\frac{1}{p} \pr{ \inrd (M_d |g|^{\frac{p-\epsilon_2}{p - r}})^\frac{p}{p-\epsilon_2}}^\pr{\frac{1}{r} - \frac{1}{p}}
\\ & \lesssim   \br{U,V}_{\Ap} ^\frac{1}{p} (\epsilon_1) ^{-\frac{1}{p}} (\epsilon_2)^{- \pr{\frac{1}{r} - \frac{1}{p}}} \|\V{f}\|_{L^p(V)} \|g\|_{L^{\frac{p}{p-r}}} ^{\frac{1}{r}} \end{align*} where $M_d$ is the dyadic maximal function.  This completes the proof.

Notice that when $r = 1$ and $p > 1$ we get that \begin{equation*} \norm{\MC{\tilde{S}}_{U, 1}}_{L^p(V)  \rightarrow L^p} \lesssim \br{U,V}_{\Ap} ^p \br{V^{-\frac{1}{p}}} _{\Apprimewk} ^\frac{1}{p} \br{U}_{\Apwk} ^\frac{1}{p'}  \end{equation*} which as mentioned before coincides with the best known A${}_p$ - A${}_\infty$ bound for sparse operators, since a sparse operator $\tilde{S}_{U, V}$ defined by \begin{align*} \tilde{S}_{U} \V{f} (x) = \sum_{L \in \L} \abs{U^\frac{1}{p} (x) \innp{ \V{f}}_L} 1_L(x) \end{align*}      can be trivially dominated by $\MC{\tilde{S}}_{U,1}$.

\end{document}